\documentclass[english,11pt]{smfart}

\usepackage{etex}

\usepackage{amsbsy}
\usepackage{amsmath,amsfonts,amssymb,amsthm,mathrsfs,mathtools}

\usepackage[a4paper,vmargin={3.5cm,3.5cm},hmargin={2.5cm,2.5cm}]{geometry}
\usepackage[font=sf, labelfont={sf,bf}, margin=1cm]{caption}
\usepackage{graphicx}
\usepackage{epsfig}
\usepackage{latexsym}
\usepackage{xcolor}
\usepackage{ae,aecompl}
\usepackage{soul,framed}
\usepackage{comment}

\usepackage{xcolor}
\usepackage[pdfpagemode=UseNone,bookmarksopen=false,colorlinks=true,urlcolor=blue,citecolor=blue,citebordercolor=blue,linkcolor=blue]{hyperref}
\usepackage{smfhyperref}
\usepackage[capitalize]{cleveref}

\usepackage{pstricks}
\usepackage{enumerate}
\usepackage{tikz,animate,media9}						
\usepackage{todonotes}
\usepackage{pifont}
\usepackage{bm,marvosym}
\usepackage{algorithm}
\usepackage{algorithmic}

\usepackage{bbm}

\newcommand{\spa}{\mathsf{span}}

\newcommand{\ndN}{\mathbb{N}}

\newcommand{\ndR}{\mathbb{R}}

\newcommand{\ndK}{\mathbb{K}}

\newcommand{\Di}{\textnormal{D}}

\renewcommand{\Pr}[1]{\mathbb{P}(#1)}

\newcommand{\Ex}[1]{\mathbb{E}[#1]}

\newcommand{\convdis}{\,{\buildrel d \over \longrightarrow}\,}

\newcommand{\convp}{\,{\buildrel p \over \longrightarrow}\,}

\newcommand{\He}{\mathrm{He}}

\newcommand{\cA}{\mathcal{A}}
\newcommand{\cB}{\mathcal{B}}
\newcommand{\cC}{\mathcal{C}}

\newcommand{\cE}{\mathcal{E}}

\newcommand{\cK}{\mathcal{K}}
\newcommand{\cL}{\mathcal{L}}

\newcommand{\cN}{\mathcal{N}}

\newcommand{\cR}{\mathcal{R}}
\newcommand{\cS}{\mathcal{S}}
\newcommand{\cT}{\mathcal{T}}
\newcommand{\cU}{\mathcal{U}}
\newcommand{\cV}{\mathcal{V}}

\newcommand{\cX}{\mathcal{X}}

\newcommand{\Cyc}{\textsc{CYC}}

\newcommand{\scX}{\mathscr{X}}

\newcommand{\scU}{\mathscr{U}}
\newcommand{\scA}{\mathscr{A}}
\newcommand{\scB}{\mathscr{B}}
\newcommand{\scC}{\mathscr{C}}
\newcommand{\scR}{\mathscr{R}}
\newcommand{\scT}{\mathscr{T}}

\newcommand{\mK}{\mathsf{K}}

\newcommand{\mS}{\mathsf{S}}

\newcommand{\mU}{\mathsf{U}}

\newtheorem{theorem}{Theorem}[section]

\newtheorem{corollary}[theorem]{Corollary}
\newtheorem{proposition}[theorem]{Proposition}
\newtheorem{lemma}[theorem]{Lemma}

\newtheorem{definition}[theorem]{Definition}
\newtheorem{remark}[theorem]{Remark}
\newtheorem{example}[theorem]{Example}

\numberwithin{equation}{section}

\keywords{random trees, unrooted plane trees}

\title{\textbf{Simply generated unrooted plane trees}}

\author{Leon Ramzews}
\thanks{The second author is supported by the Swiss National Science Foundation grant number 200020\_172515.}
\address{University of Munich}
\email{Leon.Ramzews@math.lmu.de}

\author{Benedikt Stufler}
\address{University of Zurich}
\email{Benedikt.Stufler@math.uzh.ch}

\begin{document}

\begin{abstract}
	We study random unrooted plane trees with $n$ vertices sampled according to the weights corresponding to the vertex-degrees. Our main result shows that if the generating series of the weights has positive radius of convergence, then this model of random trees may be approximated geometrically by a Galton--Watson tree conditioned on having a large random size. 
	This implies that a variety of results for the well-studied planted case also hold for unrooted trees, including Gromov--Hausdorff--Prokhorov scaling limits, tail-bounds for the diameter, distributional graph limits, and limits for the maximum degree. Our work complements results by Wang~(2016), who studied random unrooted plane trees whose diameter tends to infinity.
\end{abstract}

\maketitle

\section{Introduction}
\label{sec:intro}
The classical model of trees studied in probability theory is that of a \emph{planted plane tree}. Such a tree has a distinguished vertex called the root, and each vertex may have a linearly ordered list of children. Surely the most prominent example for a model of random plane trees is that of a Galton--Watson tree. Of particular interest is the asymptotic behaviour when conditioning this model on producing a tree with a specific number of vertices, leaves, or more generally vertices with outdegree in a fixed set, or on having a specific height, and letting this parameter tend to infinity \cite{MR3227065,MR3164755,MR1207226, MR2908619}. In a certain sense, these models may be termed \emph{simply generated}, as they fall under the more general setting of fixing a weight-sequence, assigning to each tree the product of weights corresponding to its vertex out-degrees, and sampling a tree with probability proportional to this weight from some set of plane trees. 

Apart from planted plane trees, many other types of combinatorial trees have been studied from both an enumerative and probabilistic viewpoint. Such trees may be rooted or unrooted, ordered or unordered, and labelled or unlabelled.  Drmota's book \cite{MR2484382} gives an extensive account on the subject and since then further additions to the field have been made~\cite{MR2829313,MR3050512, MR3773800,SRSAtoappear,StEJC2018,MR3573447}.


\begin{figure}[ht]
	\centering
		\centering
		\includegraphics[width=0.5\textwidth]{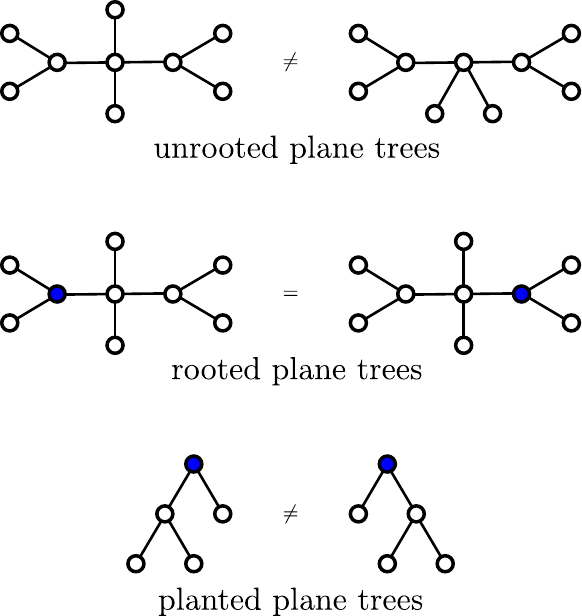}
		\caption{Different types of plane trees.}
		\label{fi:trees}
\end{figure}  

The present paper is concerned with simply generated unrooted plane trees as their number of vertices becomes large.  In the language of planar maps, whose study has received considerable attention in recent literature~\cite{MR3025391}, an \emph{unrooted plane tree} is an (unrooted) planar map with only a single face. In combinatorial terms, we may describe it as an unlabelled tree where each vertex is endowed with a cyclic ordering of its neighbourhood. By distinguishing a vertex, we may form a \emph{rooted plane tree}. Note that rooted plane trees differ from planted plane trees. In rooted plane trees, any vertex is endowed with a cyclic ordering of its neighbourhood, but in a planted plane tree it is a linear ordering on its offspring set. Figure~\ref{fi:trees} illustrates these subtle differences.

Our main result shows under minimal assumptions on the weights how we may geometrically approximate simply generated unrooted and rooted plane trees by the well studied model of simply generated planted plane trees. In fact, the only requirement we are going to make is that the generating series of the weights has positive radius of convergence. This level of generality is also one of the main difficulties in the proof. We comment on the non-analytic case in Remark~\ref{re:superexponential} below.

The approximation is very accurate, so that \emph{practically all} results available for Galton--Watson trees conditioned on having a fixed size carry over to simply generated unrooted plane trees.   This includes a variety of graph limits such as Gromov--Hausdorff--Prokhorov scaling limits and Benjamini-Schramm limits, but also tail bounds for the diameter, a central limit theorem for the maximum degree, and results for other graph theoretic parameters.

A model of random unrooted plane trees  has previously been studied by Wang~\cite{MR3573447}, who established powerful limit theorems as the diameter of these trees tends to infinity. There are some similarities in the approach of the present work  and that of \cite{MR3573447}, although the focus in the latter is on trees with a fixed diameter. Uniform random unlabelled unrooted unordered trees were studied in~\cite{SRSAtoappear} via the cycle-pointing technique, which we are not going to use in the present work. It could be applied to uniform random unrooted plane trees, but we aim for a much higher level of generality.

\section*{Outline of the paper}

In Section~\ref{sec:mainresults} we present our main theorem. In Section~\ref{sec:strategy} we describe the proof strategy and state our main lemmas. In Section~\ref{sec:applications} we state some applications of the main theorem. In particular, Subsection~\ref{sec:scaling} provides scaling limits and  tail-bounds for the diameter.  Subsection~\ref{sec:local} establishes a Benjamini-Schramm limit and a central limit theorem for the degree of a random vertex. Subsection~\ref{sec:vdegree} discusses a central limit theorem for the maximum degree and applications to other graph parameters. In Section~\ref{sec:preliminaries} we recall the required combinatorial background to prepare for the proof of the geometric approximation. In Section~\ref{sec:proof} we present the proof of our main result. The proofs of our applications are collected in Section~\ref{sec:profapp}.

\section*{Notation}
We let $\ndN$ denote the set of positive integers and set $\ndN_0 = \ndN \cup \{0\}$. The sets of positive and non-negative real numbers are denoted by $\ndR_{>0}$ and $\ndR_{\ge 0}$. Throughout, we assume that all considered random variables are defined on a common probability space.  The \emph{total variation distance} between two random variables $X$ and $Y$ with values in a countable state space $S$ is defined by
\[
d_{\textsc{TV}}(X,Y) = \sup_{\cE \subset S} |\Pr{X \in \cE} - \Pr{Y \in \cE}|.
\]
A sequence of real-valued random variables $(X_n)_{n \ge 1}$ is \emph{stochastically bounded}, if for each $\epsilon > 0$ there is a constant $M>0$ with
\[
\limsup_{n \to \infty} \Pr{ |X_n| \ge M} \le \epsilon.
\]
We denote this by $X_n = O_p(1)$. Likewise, we write $X_n = o_p(1)$ if the sequence converges to $0$ in probability. We use  $\convdis$ and $\convp$ to denote convergence in distribution and probability.
A function 
\[
h: \, \ndR_{>0} \to \ndR_{>0} \]is termed \emph{slowly varying}, if for any fixed $t >0$ it holds that
\[
\lim_{x \to \infty}\frac{h(tx)}{h(x)} = 1.
\]
For any power series $f(z)$, we let $[z^n]f(z)$ denote the coefficient of $z^n$. We also use the notation
\[
[n] = \{1, \ldots n\}
\]
for $n \ge 0$. In particular, $[0] = \emptyset$.

\section{Main result}
\label{sec:mainresults}

We let $\mathbf{w} = (\omega_i)_{i \ge 0}$ denote a fixed sequence of non-negative weights such that $\omega_0 >0$ and $\omega_k>0$ for at least one integer $k \ge 2$. The only restriction that we impose on this weight-sequence is that the series 
\[
	\Phi(z) = \sum_{k \ge 0} \omega_k z^k
\]
has positive radius of convergence $\rho_\Phi > 0$. We let $\scU_n^{}$ denote the set of unlabelled unrooted plane trees with $n$ vertices. Any tree $U \in \scU_n^{}$ receives the weight
\[
	\bar{\omega}(U) = \prod_{v \in U} \omega_{d_U(v) - 1},
\]
with the index $v$ ranging over all vertices of $U$, and $d_U(\cdot)$ denoting the degree of a vertex. 
We let $\cU_n^{}$ denote the random tree sampled from the set $\scU_n^{}$ with probability proportional to its $\bar{\omega}$-weight. Of course this is only possible when there are trees of size $n$ with positive weight. Let $\spa(\mathbf{w})$ denote the greatest common divisor of all integers $i \ge 0$ satisfying $\omega_i >0$. The following basic fact clarifies which values of $n$ we may consider. It follows from Schur's lemma, see for example \cite[Thm. 3.15.2]{MR2172781}, by noting that $\sum_{v\in U} (d_U(v) - 1) = n - 2$ for $U\in\scU_n$ and that all numbers in $\{ i / \spa(\mathbf{w}) : \omega_i > 0, ~ i\le n \}$ are relatively prime for all $n$ such that the set is non-empty.
\begin{proposition}	
	If there is a tree in the set $\scU_n^{}$ having positive weight, then it holds that 
	\[
		n \equiv 2 \mod \spa(\mathbf{w}).
	\] 
	Conversely, such a tree always exists if $n$ is large enough and belongs to this congruence class.
\end{proposition}

Likewise, for any positive integer $m$  we let $\scT_{m}^{}$ denote the set of  planted plane trees with $m$ vertices. Any such tree $T$ receives weight
\[
\omega(T) = \prod_{v \in T} \omega_{d^+_T(v)},
\]
with $d^+_T(\cdot)$ denoting the outdegree. We let $\cT_m^{}$ denote the tree sampled from $\scT_m^{}$ with probability proportional to its $\omega$-weight. This is well-defined  when \[m \equiv 1 \mod \spa(\mathbf{w})\] is large enough.


Our  main result reduces the study of the tree $\cU_n$  to that of the simply generated (planted) plane tree $\cT_m$. 
\begin{theorem}
	\label{te:main}
	Suppose that $\rho_\Phi>0$.
	Then there are constants $C,c>0$ and a random planted plane tree $\cV_n$, that is independent from the family  $(\cT_m)_m$ of simply generated planted plane trees and has stochastically bounded size
	\[
		K_n := |\cV_n| = O_p(1),
	\]
	such that the random tree $\cS_n$, constructed by connecting the root of $\cT_{n-K_n}$ and $\cV_n$ with an edge, satisfies
	\[
	d_{\textsc{TV}}(\cU_n, \cS_n) \le C \exp(-cn)
	\]
	for all $n$. Furthermore, there is a Galton--Watson tree $\cT$ that is subcritical or critical such that
	\[
	\cV_n \convdis \cT 
	\]
	as $n$ becomes large. The offspring distribution of $\cT$ may be  constructed from the weight-sequence $\mathbf{w}$ in a canonical way, see Equation~\eqref{eq:offsp} below.  
\end{theorem}
In other words, the random unrooted tree $\cU_n$ is very likely to look like the tree $\cT_{n-K_n}$ with a small tree attached to its root, to make up for the $K_n = O_p(1)$ missing vertices. 
This result proves  that \emph{almost every} asymptotic property known for simply generated plane trees also holds for simply generated unrooted plane trees. This includes Gromov--Hausdorff--Prokhorov scaling limits, as the attached tree is so small, that it does not change the global geometric  shape. It includes Benjamini-Schramm limits, because for the relevant weight-sequences a random vertex is unlikely to fall into the small attached tree or anywhere near it. For the same reason, (central) limit theorems known for the degree of a random root carry over. The approximation also preserves limits for the maximum degree and any other graph-theoretic property, that does not get heavily perturbed by the small $O_p(1)$-sized tree. As the total variational distance is exponentially small, tail bounds for the diameter and other parameters carry over as well. Our setting is also very general, since the only assumption we made on the weight-sequence is that its generating series $\Phi(z)$ has positive radius of convergence.

 A similar approximation was constructed in \cite{SRSAtoappear} in a different setting, where it was shown that random unlabelled unordered unrooted may be approximated by random unlabelled unordered rooted trees, and hence everything known (present and future) about rooted trees carries over \emph{automatically} to the unrooted model.


\section{Proof strategy}
\label{sec:strategy}

We let $\cT_n^*$ denote the random tree sampled from the set $\scT_n$ of all planted plane trees with probability proportional to its $\bar{\omega}$-weight (as opposed to the $\omega$-weight in the definition of the random tree $\cT_m$). Our first main lemma states that $\cT_n^*$ is an excellent approximation of the random tree $\cU_n$.
\begin{lemma}
	\label{le:step1}
	Suppose that $\rho_\Phi>0$. There are constants $C,c>0$ such that as unrooted trees it holds for all $n$
	\[
		d_{\textsc{TV}}(\cU_n, \cT_n^*) \le C \exp(-cn).
	\]
\end{lemma}

In order to deduce Theorem~\ref{te:main}, we consider the fringe subtree $\cT_n^{(1)}$ at the first son of the root of $\cT_n^*$ and the remaining pruned tree $\cT_n^{(2)}$. Recent results for convergent Gibbs partitions~\cite{doi:10.1002/rsa.20771} show that the maximum size of the two trees belongs to $n  + O_p(1)$.
\begin{lemma}
	\label{le:step2}
	Suppose that $\rho_\Phi>0$. Let $a,b$ be non-negative integers with $a+b =n$ such that the event $(|\cT_n^{(1)}|,|\cT_n^{(2)}|)  = (a,b)$  has positive probability. Then the conditioned pair of trees
	\[
		((\cT_n^{(1)}, \cT_n^{(2)}) \mid |\cT_n^{(1)}| = a, |\cT_n^{(2)}| = b)  
	\]
	is distributed like the pair of an independent copy of $\cT_a$ and of $\cT_b$. Moreover, if we let $\cT_n^{\mathrm{max}}$ denote the largest tree in the forest $\cT_n^{(1)}, \cT_n^{(2)}$, and $\cT_n^{\mathrm{min}}$ the smallest one, then there is a Galton--Watson tree $\cT$ that is critical or subcritical such that
	\begin{align}
		\label{eq:tvlimit}
		\cT_n^{\mathrm{min}} \convdis \cT
	\end{align}
	as $n \equiv 2 \mod \spa(\mathbf{w})$ becomes large.
	As the tree $\cT$ is almost surely finite, this implies that
	\[
		|\cT_n^{\mathrm{max}}| = n + O_p(1).
	\]
	The offspring distribution of $\cT$ is made explicit in Equation~\eqref{eq:offsp}.
\end{lemma}
Theorem~\ref{te:main} follows readily from this result. We are going to prove Lemmas~\ref{le:step1} and \ref{le:step2} in Section~\ref{sec:proof}.

\begin{remark}
	\label{re:superexponential}
	The present work focuses on the case $\rho_\Phi>0$. It is natural to wonder what happens when the radius of convergence $\rho_\Phi$ equals zero. 
	
	In fact, Lemma~\ref{le:step2} still holds in this case with the offspring distribution $\xi$ from Equation~\eqref{eq:offsp} being concentrated on $0$ and hence $\cT$ consisting almost surely of a single root vertex. This may be proved in an analogous fashion as the case $\rho_\Phi>0$ by using results for superexponential Gibbs partitions~\cite[Lem. 6.17, Thm. 6.18]{2016arXiv161202580S} instead of results for the convergent case of Gibbs partitions~\cite[Lem. 3.3, Thm. 3.1]{doi:10.1002/rsa.20771}.
	
	We conjecture that Lemma~\ref{le:step1} and hence also Theorem~\ref{te:main} still hold for $\rho_\Phi = 0$. The bounds in the proof of \cite[Thm. 6.18]{2016arXiv161202580S} may prove helpful for this, but we did not go through any details.
\end{remark}

\section{Applications}
\label{sec:applications}

In this Section we provide some applications of our main result. We collect their proofs in Section~\ref{sec:profapp}. In order to precisely state them we make use of a canonical choice for an offspring distribution $\xi$ such that for any admissible integer $m$ the simply generated tree $\cT_m$ is distributed like a $\xi$-Galton--Watson tree $\cT$ conditioned on having $m$ vertices. We recall its construction as in \cite[Thm. 7.1]{MR2908619}, but refer the reader to this source for detailed justifications.

We  set
\begin{align}
\label{eq:psi_nu}
\Psi(t) := t\Phi'(t)/\Phi(t)
\quad
\text{and}
\quad
\nu := \lim_{t \nearrow \rho_\Phi} \Psi(t) \in ]0, \infty].
\end{align}
If $\nu \ge 1$, there is a unique finite number $\tau$ with $\Psi(\tau) = 1$. If $\nu <1$, we set $\tau = \rho_\Phi$. Then the random non-negative integer $\xi$ with distribution
\begin{align}
\label{eq:offsp}
\Pr{\xi = k} = \omega_k \tau^k / \Phi(\tau), \qquad k \ge 0,
\end{align}
has mean
\begin{align}
\mu := \min(\nu, 1)
\end{align}
and variance
\begin{align}
\label{eq:sigma}
\sigma^2 := \tau \Psi'(\tau).
\end{align}
By \cite[Lem. 4.1]{MR2908619} we have for any admissible integer $m$ that the simply generated tree $\cT_m$ is distributed like the $\xi$-Galton--Watson tree $\cT$ conditioned on having $m$ vertices.
\subsection{Scaling limits and tail bounds for the diameter}
\label{sec:scaling}
Gromov--Hausdorff--Prokhorov scaling limits describe the asymptotic global geometric behaviour of a sequence of random geometric spaces. We refer the reader to the surveys by Haas~\cite{2016arXiv160507873H} and Le Gall and Miermont~\cite{MR3025391} for an overview on scaling limits of random trees, and in-depth discussions of real-trees and the Gromov--Hausdorff--Prokhorov distance $d_{\textsc{GHP}}$ on the space $\ndK$ of equivalence classes of measured compact metric spaces. A brief introduction is provided below.

\begin{theorem}
	\label{te:scaling}
	Let the random tree $\cU_n$ be endowed with the uniform measure $\nu_n$ on its leaves. Assume that $\mu = 1$.
	\begin{enumerate}
		\item Let  $\cT_{\mathrm{Br}}$ denote the (Brownian) continuum random tree introduced by Aldous~\cite{MR1207226}, and $\mu_{\mathrm{Br}}$ its probability measure on its set of leaves. If the variance $\sigma^2$ is finite, then 
\begin{align}
	\label{eq:ap1}
			\left (\frac{\sigma \cU_n}{2 \sqrt{n}}, \nu_n \right ) \convdis (\cT_{\mathrm{Br}}, \mu_{\mathrm{Br}})
\end{align}
		in the space $(\ndK, d_{\textsc{GHP}})$ as $n$ becomes large. Moreover, there are constants $C,c>0$ such that the diameter $\Di(\cU_n)$ of the tree $\cU_n$ satisfies for all $n$ and $x \ge 0$
\begin{align}
		\label{eq:ap2}
			\Pr{\Di(\cU_n) \ge x}\le C \exp(-c x^2 /n).
\end{align}
		\item Suppose that $\spa(\mathbf{w})=1$ and that the random variable $\xi$ belongs to the domain of attraction of a stable-law with index $\alpha\in]1,2[$. Let $(\cT_\alpha, \mu_\alpha)$ denote the $\alpha$-stable L\'evy tree introduced by Le Gall and Le Jan~\cite{MR1617047}. Then there is a slowly varying sequence $g(n)$ such that
\begin{align}
		\label{eq:ap3}
			\left (\frac{\cU_n}{g(n)n^{1-1/\alpha}}, \nu_n \right ) \convdis (\cT_\alpha, \mu_\alpha)
\end{align}
		in the space $(\ndK, d_{\textsc{GHP}})$ as $n$ becomes large. Moreover, for all $\delta \in ]0, \alpha[$ there are constants $C,c>0$ such that for all $n$ and $x \ge 0$
\begin{align}
			\label{eq:ap4}
			\Pr{\Di(\cU_n) \ge x} \le C \exp \left(-c \left(\frac{x}{g(n) n^{1 - 1/\alpha}} \right)^{\delta} \right).
\end{align}
	\end{enumerate}
\end{theorem}
Limits of this form have far-reaching consequences, see for example the recent work~\cite{doi:10.1137/1.9781611975031.60}.
Let us very briefly explain relevant notation used in Theorem~\ref{te:scaling}. For any metric space $(Z,d_Z)$ and any  real number $\beta > 0$ we let $\beta Z$ denote the rescaled space $(Z, \beta d_Z)$.  The \emph{Hausdorff-distance} defines a metric on the collection of compact subsets of $Z$. It is given by
\[
	d_{\textsc{H}}(A,B) = \inf\{\epsilon > 0 \mid A \subset U_{\epsilon}(B), B \subset U_{\epsilon}(A)\} 
\]
with $A^{\epsilon}$ denoting the $\epsilon$-hull of the subset $A$. That is, the set of all points in $Z$ with distance less than $\epsilon$ from $A$. The \emph{Prokhorov-distance} metrizes weak convergence of probability  measures on the Borel $\sigma$-field $\cB(Z)$ of $Z$.  For any two such measures $\mathbb{P}$ and $\mathbb{P}'$ it is given by
\[
	d_{\textsc{P}}(\mathbb{P}, \mathbb{P}') = \inf\{ \epsilon >0 \mid \forall A \in \cB(X): \mathbb{P}(A) \le \mathbb{P}'(A^\epsilon) + \epsilon,\mathbb{P}'(A)  \le \mathbb{P}(A^\epsilon) + \epsilon \}.
\]

Let  $(X, d_X)$, $(Y, d_Y)$ be compact metric spaces endowed with Borel probability measures $\mathbb{P}_X$ and $\mathbb{P}_Y$. Unless these spaces are subspaces of a common space, we cannot measure their distance using the Hausdorff or Prokhorov metric. The natural solution is to consider embeddings.
The \emph{Gromov--Hausdorff} distance of  $(X, d_X)$ and $(Y, d_Y)$ is given by the minimal Hausdorff distance of isometric copies of $X$ and $Y$ in a common space. That is
\[
	d_{\textsc{GH}}(X,Y) = \inf_{\phi_X, \phi_Y} d_{\textsc{H}}(\phi_X(X),\phi_Y(Y))
\]
with the indices ranging over all isometric embeddings $\phi_X: X \to Z$ and $\phi_Y: Y \to Z$ for all choices of metric spaces $(Z, d_Z)$. Letting $\phi^*_X \mathbb{P}_X$ and $\phi^*_Y \mathbb{P}_Y$ denote the push-forwards of the measures, the \emph{Gromov--Hausdorff--Prokhorov} distance is defined by
\[
d_{\textsc{GHP}}((X, \mathbb{P}_X),(Y, \mathbb{P}_Y)) = \inf_{\phi_X, \phi_Y} \min(d_{\textsc{H}}(\phi_X(X),\phi_Y(Y)), d_{\textsc{P}}(\phi^*_X \mathbb{P}_X,\phi^*_Y \mathbb{P}_Y)).
\]
The GHP distance satisfies the axioms of a pre-metric and hence induces a metric on the collection of equivalence classes $\ndK$ of measured compact metric spaces.  Here two such spaces are equivalent if and only if there is a measure-preserving isometry between them. Here we admittedly have swept a small issue under the table, as these classes are not sets and hence technically their collection is not a well-defined object. But this issue is easily resolved by working with representatives instead, compare with \cite[Remark 7.2.5]{MR1835418}.

The metric space $(\ndK, d_{\textsc{GHP}})$ is known to be complete and separable \cite[Thm. 2.3]{MR3035742}, and hence we use the classical notions of  weak convergence of Borel-measures on Polish spaces \cite[Ch. 1]{MR0310933}.

\subsection{Benjamini-Schramm limits and degree distribution}
\label{sec:local}

The term \emph{Benjamini-Schramm limit} refers to the local weak limits of a sequence of finite random graphs with respect to an uniformly at random sampled point.  The name is used in honour of the authors of \cite{MR1873300}.

\begin{theorem}
	\label{te:bslimit}
	Suppose that $\mu=1$ and $\sigma^2<\infty$. 
	\begin{enumerate}
		\item 	Let $\cT^\bullet$ denote the Benjamini-Schramm limit of the conditioned Galton--Watson trees $(\cT_m)_{m \ge 1}$ established by Aldous~\cite{MR1102319}. Then $\mU_n$ converges also in the Benjamini-Schramm sense toward $\cT^\bullet$. 
		\item If $N_d$ denotes the number of vertices of degree $d \ge 1$ in the tree $\mU_n$, then
	\[
		\frac{N_d - \Pr{\xi = d-1}}{\sqrt{n}} \convdis \cN(0, \sigma_d^2)
	\]
	for some $0<\sigma_d <\infty$.
	\end{enumerate}
\end{theorem}
The distribution of $\cT^\bullet$ is made explicit in \cite[Remark 7.13, Remark 5.3]{MR2908619}. We very briefly recall the notion of local weak convergence used in Theorem~\ref{te:bslimit} and refer the reader to Curien's  notes \cite{clecture} for a more in-depth  treatment. 
A \emph{locally finite} graph may have infinitely many vertices, but each has finite degree. Any two such graphs $G$ and $H$ with distinguished vertices $v_G$ and $v_H$ are considered as \emph{isomorphic}, if there is a bijection between their vertex sets that preserves the incidence relation and maps the root vertices to each other. This is denoted by $(G, v_G) \simeq (H, v_H)$. For any $\ell \ge 0$, the subgraph consisting of all vertices with graph distance at most $\ell$ from $v_G$ is denoted by $U_\ell(G, v_G)$ and called the \emph{$\ell$-neighbourhood}. We consider the $\ell$-neighbourhood as a rooted graph. The \emph{local distance} of the two rooted graphs $H$ and $G$ is defined by
\[
	d_{\textsc{L}}( (G, v_G), (H,v_H) ) = 1 / (1 + \sup\{ \ell \ge 0 \mid U_\ell(G,v_G) \simeq U_\ell(H, v_H)\}).
\]
This yields a metric on the collection of (representatives of) isomorphism classes of rooted locally finite graphs, which is known to be complete and separable. A sequence of random rooted graphs $(G_n, v_n)$ converges in the \emph{local weak sense} toward a random limit graph $(G, v)$, if for each $\ell \ge 0$ and each rooted graph $(Q, v_Q)$ it holds that
\[
	\Pr{U_\ell(G_n, v_n) \simeq (Q, v_Q)} \to \Pr{U_\ell(G,v) \simeq (Q, v_Q)}
\]
as $n$ becomes large.

\subsection{The maximum degree and other parameters}
\label{sec:vdegree}
Theorem~\ref{te:main} reduces the study of the extremal vertex degree sizes of the unrooted tree $\cU_n$ to those of the random tree $\cS_n$. It is clear that central limit theorems or laws of large numbers for the degree sizes in the simply generated tree $\cT_m$ also hold for the tree $\cT_{K_n}$, since $K_n = n + O_p(1)$. The trees $\cT_{K_n}$ and $\cS_n$ differ only by a small tree with stochastically bounded size that gets attached to the root of $\cT_{K_n}$. Consequently, available results for the maximum degree of $(\cT_m)_{m \ge 1}$ automatically carry over to the random unrooted tree $\cU_n$. In particular, we obtain by recent results of Kortchemski~\cite[Thm. 1]{MR3335012} a central limit theorem for the largest degree in the subcritical setting.
\begin{corollary}
	Suppose that $\mu < 1$ and
	\[
		\Pr{\xi = k} = f(k) k^{-\beta}
	\]
	for some slowly varying function $f$ and a parameter $\beta>2$. Set \[\alpha=\min(2, \beta-1)\] and let $(Y_t)_{t\ge1}$ denote a spectrally positive L\'evy process with Laplace exponent
		\[
		\Ex{\exp(-\lambda Y_t)} = \exp(t \lambda^{\alpha}).
		\]
		Then there exists a slowly varying function $g$ such that the maximum degree $\Delta(\cU_n)$ satisfies
		\[
		\frac{(1-\nu)n - \Delta(\cU_n)}{g(n)n^{1/\alpha}} \convdis Y_1.
		\]
	The second largest degree $\Delta_2(\cU_n)$ satisfies
	\[
		\Delta_2(\cU_n) = O_p(g(n)n^{1/\alpha}).
	\]
\end{corollary}
Informally speaking, pretty much all ``interesting" properties known for simply generated planted plane trees carry over to the random tree $\cU_n$ using the approximation in Theorem~\ref{te:main} (and possibly Lemma~\ref{le:step2}). See in particular Janson's comprehensive survey \cite{MR2908619} for a wealth of further results that may be transferred. 

\section{Preliminaries}
	\label{sec:preliminaries}

\subsection{Subexponential power series}

Subexponential sequences  were studied by Chover, Ney and Wainger \cite{MR0348393}, Embrechts \cite{MR714482}, and Embrechts and Omey~\cite{MR772907}. Up to tilting and rescaling, these sequences correspond to subexponential densities of random variables with values in a lattice, and hence belong to the context of heavy-tailed and subexponential distributions, see the book by Foss, Korshunov, and Zachary \cite{MR3097424}.  

\begin{definition}
	Let $d \ge 1$ be an integer. A power series $g(x) = \sum_{k \ge 0} g_k x^k$ with non-negative coefficients and radius of convergence $\rho >0$ is subexponential with span $d$, if $g_k=0$ whenever $k$ is not divisible by $d$, and
	\begin{align}
	\label{eq:condition}
	\frac{g_k}{g_{k+d}} \sim \rho^d, \qquad \frac{1}{g_k}\sum_{i+j=k}g_ig_j \sim 2 g(\rho) < \infty
	\end{align}
	as $k \equiv 0 \mod d$ becomes large.
\end{definition}
Following lemma is useful for deriving the behaviour of a randomly stopped random walk with i.i.d. steps (after centralizing the coefficients of $g(x)$ properly), and in our case it will prove handy as enumerative tool.
\begin{lemma}[{\cite[Thm. 4.8,~4.30]{MR3097424}}]
	\label{lem:subexp_composition}
	If $g(x)$ is subexponential with span $d$ and radius of convergence $\rho>0$, and $f(x)$ is a non-constant power series with non-negative coefficients that is analytic at $\rho$, then $f(g(x))$ is subexponential with span $d$ and radius of convergence $\rho$. Further, as $n \equiv 0 \mod d$ becomes large,
	\[
		[x^n] f(g(x)) \sim f'(g(\rho)) [x^n]g(x).
	\]
\end{lemma}
It was observed in \cite{{doi:10.1002/rsa.20771}} by building on  results for simply generated trees~\cite{MR2908619}, that power series satisfying a simple recursive relation are always subexponential up to a shift.

\begin{lemma}[{\cite[Lem. 3.3]{doi:10.1002/rsa.20771}}]
	\label{le:partf}
	Let $g(x)$ and $f(x)$ be power series with non-negative coefficients such that
	\[
		g(x) = x f(g(x)).
	\]
	If the series $f(x)$ has positive radius of convergence and satisfies $f(0) >0$ and $[x^k]f(x)>0$ for at least one $k \ge 2$, then $g(x)/x$ is subexponential with span $d$ for some $d \ge 1$.
\end{lemma}

\begin{lemma}
	\label{le:coeff_comparison_subexp}
Let $g(x) = \sum_{k \ge 0}g_k x^k$ and $f(x) = \sum_{k \ge 0}f_k x^k$ be power series with non-negative coefficients. Suppose that the radii of convergence $\rho_g$ and $\rho_f$ of $g(x)$ and $f(x)$ satisfy $\rho_g > \rho_f>0$. If $d \ge 1$ is an integer with
\[
	\frac{f_n}{f_{n+d}} \sim \rho_f^d
\]
as $n \equiv 0 \mod d$ becomes large, then there exist constants $C,c>0$ such that
\[
	\frac{g_n}{f_n} \le C \exp(-cn)
\]
holds for all $n \equiv 0 \mod d$ with $f_n>0$.
\end{lemma}
\begin{proof}
The assumption $f_n / f_{n+d} \sim \rho_f^d$ implies that the radius of convergence of $f(x)$ does not change if we restrict to coefficients $f_n$ with index in the lattice $d \ndN_0$. It also implies that $f_{kd}>0$ for all but finitely many integers $k \ge 0$. Hence without loss of generality we may assume that $d=1$ and $f_n >0$ for all $n \ge 0$.

Our assumption on $f(x)$ now reads $f_n/f_{n+1} \sim \rho_f$. Hence for any $\delta>0$ there is an integer $n_0\ge 0$ such that
\[
	\frac{f_{n_0}}{f_{n_0 + k}} \le (1 + \delta)^k \rho_f^k
\]
for all $k \ge 0$. Thus
\begin{align*}
	\sum_{n\ge 0} \frac{g_n}{f_n}(1+\delta)^n
	&= \sum_{0\le n < n_0} \frac{g_n}{f_n}(1+\delta)^n
	+ (1+\delta)^{n_0} \sum_{k\ge 0} 
	\frac{g_{n_0+k}}{f_{n_0+k}} (1+\delta)^{k} \\
	&\le \sum_{0\le n < n_0} \frac{g_n}{f_n}(1+\delta)^n
	+ \frac{(1+\delta)^{n_0}}{f_{n_0}} \sum_{k\ge 0} 
	g_{n_0+k} \left( (1 + \delta)^2 \rho_f \right)^{k}.
\end{align*}
Since $\rho_g > \rho_f$ we may choose $\delta>0$ small enough such that $(1 + \delta)^2 \rho_f < \rho_g$. With this choice of $\delta$ it follows that
\[
\sum_{n\ge 0} \frac{g_n}{f_n}(1+\delta)^n < \infty.
\]
Consequently, there are constants $C,c>0$ such that
\[
\frac{g_n}{f_n} \le C \exp(-cn)
\]
for all $n \ge 0$.
\end{proof}

\subsection{Combinatorial Classes}
	\label{sec:comb_classes}

In order to understand how symmetries of random unrooted plane trees typically behave, we will make use of some enumerative results. We aim to recall just enough combinatorial background for the non-specialist.  A comprehensive survey of the theory is given in the books \cite{MR2483235, MR1629341}.

A \emph{weighted combinatorial class} is given by a set $\scC^{\omega}$ of countably many objects equipped with a size function \[\lvert\cdot\rvert:\scC^\omega\to\ndN_0\] and a weight function \[\omega:\scC^\omega\to\ndR_{\ge0}.\] Additionally, the subset $\scC^{\omega}_n \subset \scC^{\omega}$ containing all $n$-sized objects in $\scC^{\omega}$ is required to satisfy \[\sum_{C \in \scC^\omega_n} \omega(C) < \infty\] for all $n\in\ndN_0$. Each $C\in\scC^{\omega}_n$ is said to be comprised of $n$ \emph{atoms}. There are two cases to be set apart:
\begin{enumerate}
\item Each atom bears a distinct label and the class is termed \emph{labelled}.
 Then the \emph{(weighted exponential) generating series} of $\scC^{\omega}$ is the formal power series
\[
	\scC^{\omega}(x)
	:= \sum_{C\in\scC^{\omega}} \frac{\omega(C)}{\lvert C\rvert!} x^{\lvert C \rvert}.
\]
\item Whenever atoms are not distinguishable, we say that the class is \emph{unlabelled}. For this section we denote such classes by $\tilde{\scC}^{\omega}$ and the resulting \emph{(weighted ordinary) generating series} is defined as the formal power series
\[
	\tilde{\scC}^{\omega}(x)
	:= \sum_{C\in\scC^{\omega}} \omega(C) x^{\lvert C \rvert}.
\]
\end{enumerate}
Without loss of generality we will assume that any object $C\in\scC^{\omega}_n$ has labels in the set $[n]$ and that $\scC^{\omega}$ is defined in a coherent way such that it contains any possible relabelling of $C$ with labels in $[n]$. For $\sigma:[n]\to[n]$ and $C\in\scC^{\omega}_n$ we denote by $\sigma.C$ the object obtained by replacing the label $i_v\in[n]$ of the atom $v$ of $C$ by $\sigma(i_v)$ for all atoms $v\in C$. \\
We may alternatively view $\tilde{\scC}^{\omega}$ as the set of equivalence classes under the relation which terms two objects $C,C'\in\scC^{\omega}_n$ \emph{isomorphic} if and only if one object is obtained by relabelling the other one, i.e. there exists a permutation $\sigma:[n]\to[n]$ such that $\sigma.C = C'$. Regarding this, it makes sense to impose that any labelled element in a equivalence class $\tilde{C}\in\tilde{\scC}^{\omega}$ receives the same weight $\omega(\tilde{C})$.
\subsubsection{Cycle index sum}
Let $C\in\scC^{\omega}$ and $\sigma$ be a permutation on the label set of $C$ such that $\sigma.C=C$. Then $\sigma$ is called \emph{automorphism} of $C$. Note that any object has at least one automorphism, namely the identity. The class of \emph{symmetries} of $\scC^{\omega}$ is defined as the collection of objects paired with an automorphism, that is
\[
	\mathrm{Sym}(\scC^{\omega})
	:= \{ (C,\sigma) : C\in\scC^{\omega}\text{ and } \sigma.C = C\}.
\]
Any permutation $\sigma$ can be decomposed into disjoint cycles and we denote by $\sigma_i$ the number of cycles of length $i$ in $\sigma$. In particular, $\sigma_1$ counts the number of fixpoints. The \emph{cycle index sum} of a class $\scC^{\omega}$ is then defined as the formal power series 
\[
	Z_{\scC^{\omega}}(s_1,s_2,\dots)
	:= \sum_{k\ge 0}\sum_{(C,\sigma)\in\mathrm{Sym}(\scC^{\omega}_k) }
	\frac{\omega(C)}{k!} s_1^{\sigma_1}
	\cdots s_k^{\sigma_k}.
\]	
Considering symmetries is useful, as it provides a way of counting orbits.
\begin{lemma}[{\cite[Lem. 1]{MR2810913}}]
\label{le:unlabelled_n!_symmetries_and_ogs_in_cycle_index_sum}
	For any $\tilde{C}\in\tilde{\scC}^{\omega}$ there are precisely $\lvert C\rvert!$ many symmetries $(C,\sigma)$ in $\mathrm{Sym}(\scC^{\omega})$ such that $C\in\tilde{C}$. Consequently,
\[
	\tilde{\scC}^{\omega}(x)
	= Z_{\scC^{\omega}}(x,x^2,x^3,\dots).
\]
\end{lemma}

\begin{example}
In the context of our intended applications, where we consider a weight sequence $\mathbf{w} = (\omega_i)_{i \ge 0}$ of non-negative real numbers, $\scT^{\omega}$ is typically a class (of labelled or unlabelled, planted or plane, rooted or unrooted) trees and, for $T\in\scT^{\omega}$,
\[
	\omega(T)
	:= \prod_{v\in T} \omega_{f(v)}
\]
for some arbitrary mapping $f$ from the vertex set of $T$ into $\ndN_0$. For example $f(v):=d^+_T(v)$ may count the out-degree of a given vertex $v$. Furthermore, the size of a tree will be given by its number of vertices.
\end{example}
\subsubsection{Constructions}
In this section we present several constructions designed for obtaining more complex classes out of simpler ones. Let $\scA^{\omega}$ and $\scB^{\nu}$ be labelled classes and their unlabelled counterparts be denoted as before. For any arbitrary class $\scC^\omega$, we define the set $\scC^\omega[P]$, which contains all objects in $\scC^\omega$ of size $\lvert P\rvert$ relabelled canonically by a given finite set of labels $P\subset\ndN$. 
\subsubsection*{Product}
The \emph{product} of objects $A \in \scA^\omega$ and $B \in \scB^\nu$ is the tuple $(A,B)$ where each atom is relabelled according to labels in $\lvert A \rvert + \lvert B \rvert$. Formally,
\[
	\scA^{\omega} \cdot \scB^\omega
	:= \{ (A,B) : A \in \scA^\omega[P_1],~ B \in \scB^\nu[P_2],~ (P_1,P_2)
	\text{ is a partition of } [\lvert A \rvert + \lvert B \rvert] \}.
\]
By discarding the labels, the product of the \emph{unlabelled} classes $\tilde{\scA}^{\omega}$ and $\tilde{\scB}^{\nu}$ is given by the set theoretic product
\[
	\tilde{\scA}^{\omega}\cdot \tilde{\scB}^{\nu} 
	:= \{ (A,B) : A\in \tilde{\scA}^{\omega},~ B\in \tilde{\scB}^{\nu}\}.
\]
In both cases, the size function is the canonical extension $\lvert (A,B) \rvert := \lvert A \rvert + \lvert B \rvert$ of the size functions of the underlying classes, such as the weight of $(A,B)$ being given by $\omega(A)\cdot \nu(B)$. As a straight-forward consequence the generating series satisfy
\begin{align}
	\label{eq:gen_fct_product}
	(\scA^\omega \cdot \scB^\nu)(x)
	= \scA^\omega(x) \cdot \scB^\nu (x)
	\quad
	\text{and}
	\quad
	(\tilde{\scA}^{\omega}\cdot \tilde{\scB}^{\nu})(x)
	= \tilde{\scA}^{\omega}(x) \cdot \tilde{\scB}^{\nu}(x).
\end{align}
\subsubsection*{Composition}
For this construction we assume that $\scB^{\nu}_0 = \emptyset$. The composition $\scA^{\omega} \circ \scB^{\nu}$ is the set of objects obtained by picking $A\in\scA^{\omega}$ and replacing every atom in $A$ by an entire object from $\scB^{\nu}$ relabelled properly according to the compound size.
In other words, $\scA^{\omega} \circ \scB^{\nu}$ contains all sequences of the form
\[
	(A,B_1,\dots,B_k)_{\simeq},
	\quad
	A\in \scA^{\omega}_k,~
	B_i \in \scB^{\nu}[P_i],~
	1\le i\le k,
\]
where $k\ge 0$ and $(P_1,\dots,P_k)$ is a partition of $[\sum_{1\le i \le k}\lvert B_i \rvert]$. The relation ``$\simeq$" terms two sequences $(A,B_1,\dots,B_k)$ and $(A',B_1',\dots,B_k')$ in $\scA^{\omega} \circ \scB^{\nu}$ isomorphic if $A=A'$ and for any permutation $\sigma:[k]\to[k]$ such that $\sigma.A=A$ it holds $B_{\sigma(i)}=B_i'$ for $1\le i\le k$. Hence, any $M\in\scA^{\omega} \circ \scB^{\nu}$ possesses a \emph{core structure} $A\in\scA^{\omega}$ and \emph{components} $B_1,\dots,B_{\lvert A\rvert}\in\scB^{\nu}$. The size is then given by $\lvert M \rvert := \sum_{1\le i\le \lvert A \rvert}\lvert B_i \rvert$ and the compound weight by $\omega(A)\prod_{1\le i\le \lvert A \rvert} \nu(B_i)$. This implies that the generating series is
\[
	(\scA^{\omega}\circ\scB^{\nu})(x) 
	= \scA^{\omega}(\scB^{\nu}(x)).
\]
The unlabelled composition $\tilde{\scA}^{\omega} \circ \tilde{\scB}^{\nu}$ is defined equivalently, except for skipping the part where the objects are relabelled. Due to possible symmetries appearing when removing the labels, the generating series satisfies a more complex formula, namely
\begin{align}
	\label{eq:cycle_index_sum_comp}
	(\tilde{\scA}^{\omega} \circ \tilde{\scB}^{\nu})(x)
	= Z_{\scA^{\omega}}(\tilde{\scB}^{\nu}(x),\scB^{\nu^2}(x^2),\scB^{\nu^3}(x^3),\dots),
\end{align}
where $\scB^{\nu^k}$ denotes the class with weight function $\nu(B)^k$ for $B\in\scB^{\nu^k}$. This is stated in \cite[Prop. 11]{MR1629341}.
\subsubsection*{Cycle}
The labelled class $\Cyc^{\kappa}$ contains all bijections that are cycles and by convention the empty set is mapped onto itself. The size of a cycle is its length and the weight of a cycle $\tau\in\Cyc^{\kappa}$ depends only on the length, i.e. $\kappa(\tau)=\kappa_{\lvert\tau\rvert}$ for some non-negative real-valued sequence $(\kappa_k)_{k\ge 0}$ with $\kappa_0=1$. This follows from the fact that all cycles of the same length are isomorphic.\\
The symmetries of cycles may be described explicitly. It is elementary that any automorphism of a cycle $\tau$ with length $k\ge 1$ must be of the form $\tau^i$ for $0\le i \le k-1$. The disjoint cycles in $\tau^i$ all have the same length $k / \gcd(i,k)$, and $\tau^i$ has $\gcd(i,k)$ of them. There are precisely $(k-1)!$ cycles on a fixed $k$-element set. This yields
\begin{align}
\label{eq:weighted0}
	Z_{\Cyc^\kappa}(s_1, s_2, \ldots) 
	= 1 + \sum_{k \ge 1} \frac{\kappa_k}{k} \sum_{i=0}^{k-1} s_{k / \gcd(i,k)}^{\gcd(i,k)}.
\end{align}
We let $\varphi$ denote Euler's totient function. That is, for each $d \ge 1$ the number $\varphi(d)$ counts the integers that are relatively prime to $d$. For each $k \ge 1$, there are precisely $\varphi(d)$ many integers $1 \le i \le k$ such that $k / \gcd(i,k) = d$. Hence Equation~\eqref{eq:weighted0} may be rephrased by
\begin{align}
	\label{eq:weighted1}
	Z_{\Cyc^\kappa}(s_1, s_2, \ldots) 
	= 1 + \sum_{d \ge 1} \frac{\varphi(d)}{d} \sum_{j \ge 1} \frac{\kappa_{jd}}{j} s_d^j.
\end{align}
We denote by the composition $\Cyc^{\kappa}\circ \scB^{\nu}$ the labelled class containing all cyclic orderings of objects in $\scB^{\nu}$ relabelled properly according to their compound size. \\
Formally, $\Cyc^{\kappa}\circ \scB^{\nu}$ contains all sequences of the form
\[
	(C,B_1,\dots,B_k)_{\simeq}, 
	\quad C\in\Cyc^{\kappa}_k,~ B_i \in \scB^{\nu}[P_i], ~1\le i\le k,
\]
where $(P_1,\dots,P_k)$ is a partition of $[\sum_{1\le i\le k}\lvert B_i \rvert]$ and the relation ``$\simeq$" is as before. The size of such an object is given by $\sum_{1\le i\le k}\lvert B_i \rvert$ and the weight by $\kappa(C)\prod_{1\le i\le k}\nu(B_i)$.

\subsection{Boltzmann distribution}
Given the (labelled) weighted combinatorial class $\scC^{\omega}$ and a parameter $y$ such that $0<\scC^{\omega}(y)<\infty$ we define the corresponding Boltzmann probability measure by
\[
	\Pr{\cC = C} 
	:= \frac{\omega(C)y^{\lvert C\rvert}}{\lvert C\rvert!\scC^{\omega}(y)},
	\quad C\in\scC^{\omega}.
\]
The random variable $\cC$ is taking values in the entire space $\scC^{\omega}$ and when conditioning on having size $n$ we obtain the object $\cC_n$ drawn proportional to its $\omega$-weight from all objects in $\scC^{\omega}$ of size $n$. That is
\[
	\Pr{\cC_n = C}
	= \frac{w(C)}{\sum_{C'\in\scC^{\omega}_n} w(C')},
	\quad
	C \in \scC^{\omega}_n.
\]
Let $\scX^2$ denote the unique class consisting of a single ordered pair of atoms that receives weight 1 and $\scB^{\nu}$ be a labelled class such that $\scB^{\nu}_0 =\emptyset$. Denote by $\cX_n$ the $(\scX^2\circ\scB^{\nu})_n$-valued random variable drawn proportional to its compound weight. Later we will need the distribution of the remainder $\cR_n$, which is obtained by removing ``the" largest component of $\cX_n$ (if the two components of $\cX_n$ are equally sized, we pick an arbitrary one). Therefore, we state a far simplified version of \cite[Thm. 3.4]{doi:10.1002/rsa.20771} (here $\scX^2$ is replaced by an arbitrary combinatorial class).
\begin{lemma}
\label{lem:giant_component}
Suppose that $x^{-1} \scB^{\nu}(x)$ is subexponential with radius of convergence $\rho$ and span $d$. Then
\[
	\cR_n \overset{d}{\to} \cR
\]
as $n\equiv 2\mod d$ becomes large, where $\cR$ is given by
\[
	\Pr{\cR = B} 
	= \frac{\nu(B) \rho^{\lvert B \rvert}}{\lvert B \rvert ! \scB^{\nu}(\rho)},
	\quad
	B\in\scB^{\nu}.
\]
\end{lemma}
This implies that, letting $n$ tend to infinity on a properly chosen lattice, there emerges one giant component in $\cX_n$ containing all but a stochastically bounded number of atoms.

\section{Proof of the main result}
	\label{sec:proof}
We briefly recapitulate the assumptions as well as the notation needed for the entire proof. \\
Our only restriction is that a given weight sequence $\mathbf{w}=(\omega_k)_{k\ge 0}$ is such that $\omega_0>0$ and $\omega_k>0$ for some $k\in\ndN$. Further, $\Phi(z) := \sum_{k\ge 0} \omega_k z^k$  is assumed to have finite radius of convergence $\rho_{\Phi}>0$. \\
For the remaining proof section it will prove handy to transition into the language of combinatorial classes introduced in Section \ref{sec:comb_classes}. More precisely, we denote by $L(\scU)^{\omega}$ and $L(\scT)^{\bar{\omega}}$ the classes of labelled unrooted plane trees and of labelled planted trees, respectively, where $\omega$ and $\bar{\omega}$ are the  weight functions given by
\[
	\omega(T) = \prod_{v\in T} \omega_{d^+_T(v)} 
	\quad
	\text{or}
	\quad
	\overline{\omega}(T) = \prod_{v\in T} \omega_{d_T(v)-1}
\]
for any arbitrary tree $T$.
With this at hand we define the unlabelled counterparts $\scU^{\omega}=\widetilde{L(\scU)^{\omega}}$ and $\scT^{\bar{\omega}}=\widetilde{L(\scT)^{\bar{\omega}}}$. \\
For an arbitrary class (labelled or unlabelled) $\scC^{\nu}$ we denote by $\cC_n\sim\scC^{\nu}$ the random variable drawn from all objects in $\scC^{\nu}$ of size $n$ proportional to its $\nu$-weight, i.e.
\[
	\Pr{\cC_n = C} 
	=\frac{\nu(C)}{\sum_{C'\in\scC^{\nu}_n} \nu(C')},
	\quad
	C\in\scC^{\nu}_n.
\]
In this fashion let 
\begin{enumerate}
\item $\cU_n\sim\scU^{\overline{\omega}}$ be drawn proportional to its $\overline{\omega}$-weight from all unlabelled unrooted plane trees of size $n$,
\item $\cT_n\sim\scT^{\overline{\omega}}$ proportional to its $\overline{\omega}$-weight and
\item $\cT_n^*\sim\scT^{\omega}$ proportional to its $\omega$-weight from all unlabelled planted trees  of size $n$.
\end{enumerate}

\subsection{Galton-Watson trees and tilting}
Recall the notation from Equations~\eqref{eq:psi_nu}--\eqref{eq:sigma}. For any plane tree $T$ it holds that
\begin{align}
\label{eq:taq}
\Pr{\cT = T} = \frac{\omega(T) \rho_{\scT^\omega}^{|T|}} {\scT^\omega(\rho_{\scT^\omega})}
\end{align}
with $\rho_{\scT^\omega} = \tau / \Phi(\tau)$ denoting the radius of convergence of the unique power series $\scT^\omega(x)$ satisfying
\begin{align}
\label{eq:taq2}
\scT^\omega(x) = x \Phi(\scT^\omega(x)).
\end{align}
See \cite[Remark 3.2 and Remark 7.5]{MR2908619} for a justification of these facts. As we shall see in the proof of Lemma~\ref{le:step2}, the random tree $\cT$ is also the distributional limit in \eqref{eq:tvlimit}.
\begin{remark}
\label{rem:tilting}
The preceding text implies that we may without loss of generality assume that $\mathbf{w}$ is a probability weight sequence and $\rho_{\scT^{\omega}} = 1$ whenever we sample objects conditioned on having a certain size. This is due to the fact that, for $a,b>0$, the rescaled sequence $\tilde{\mathbf{w}}:=(ab^k\omega_k)_{k\ge 0}$ yields for any tree $T \in \scT^{\omega}_n$
\[
	\tilde{\omega}(T)
	:= \prod_{v\in T} \tilde{\omega}_{d^+(v)}
	= a^n b^{\sum_{v\in T}d^+(v)} \omega(T)
	= a^n b^{n-1} \omega(T).
\]
Then, both sequences induce the same conditioned probability
\begin{align*}
	\Pr{\tilde{\cT} = T \mid \lvert \tilde{\cT} \rvert = n} 
	= \frac{\tilde{\omega}(T)}
	{\sum_{T'\in \scT^{\tilde{\omega}}} \tilde{\omega}(T) }
	= \frac{\omega(T)}{\sum_{T'\in \tilde{\omega}} \omega(T) }
	= \Pr{\cT = T \mid \lvert \cT \rvert = n}.
\end{align*}
Further, choosing $a=\Phi(\tau)^{-1}$ and $b=\tau$ we obtain
\begin{align*}
	\scT^{\tilde{\omega}}(x)
	&= \sum_{T\in\scT^{\tilde{\omega}}} \tilde{\omega}(T) x^{\lvert T \rvert}
	= \sum_{T\in\scT^{\omega}} \Phi(\tau)^{-\lvert T \rvert }
	\tau^{\lvert T \rvert -1} \omega(T) x^{\lvert T \rvert} \\
	&= \tau^{-1} \sum_{T\in\scT^{\omega}} \rho_{\scT^{\omega}}^{\lvert T\rvert}
	\omega(T) x^{\lvert T \rvert}
	= \frac{\scT^{\omega}(\rho_{\scT^{\omega}} x)}{\scT^{\omega}(\rho_{\scT^{\omega}})}
\end{align*}
and we readily deduce that $\rho_{\scT^{\tilde{\omega}}}=1$.
\end{remark}

\subsection{The geometric approximation (Proof of Lemma \ref{le:step2})}
Our first aim is to prove Lemma \ref{le:step2}. Therefore, we observe that there exists a weight-preserving bijection between trees in $\scT^{\overline{\omega}}$ and ordered pairs of trees in $\scT^{\omega}$, i.e. the random object $\cT_n^{*}$ can be decomposed into $(\cT_{n_1},\cT_{n_2})$ such that $n_1 + n_2 = n$. To see that we simply split up $T\in\scT^{\overline{\omega}}$ into $(T^{(1)},T^{(2)})$, where $T^{(1)}$ is the fringe subtree at the first son of the root of $T$ and $T^{(2)}$ is the remaining pruned tree. Clearly, this guarantees that the roots of $T^{(1)}$ and $T^{(2)}$ receive weights according to their outdegree, by which we obtain trees in $\scT^{\omega}$. (Note that for any non-rooted node $v$ in an arbitrary rooted tree $T$ it holds that $d_T(v)-1=d_T^+(v)$.) On the other hand, given $(T^{(1)},T^{(2)})\in \scT^\omega \cdot \scT^\omega$, connect the root of $T^{(1)}$ to the root of $T^{(2)}$ as the leftmost son to obtain an object in $\scT^{\overline{\omega}}$. By this procedure the neighbourhood of both roots increases by one, leading to the proper weighting. Formally, this identity allows us to interpret $\cT_n^*$ as a
\begin{align}
	\label{eq:relation_weighting_gen_series_planted}
	\scT^\omega \cdot \scT^\omega
	\simeq
	\scX^2 \circ \scT^\omega 
\end{align}
composite object, with $\scX^2$ denoting the unique class consisting of a single ordered pair of atoms that receives weight $1$. Let 
\[
	\scT^\omega(x) := \sum_{T\in\scT^{\omega}} \omega(T) x^{\lvert T \rvert}
\]
be the ordinary generating function of $\scT^{\omega}$, which is known to satisfy the equation
\[
	\scT^\omega(x) = x \Phi(\scT^\omega(x)),
\]
see for example \cite[Sec. 3.1.4]{MR2484382}. Consequently, the assumptions of Lemma \ref{le:partf} are fulfilled and $\scT^\omega(z)/z$ is subexponential with $d = \spa(\mathbf{w})$.\\
The class $\scT^\omega$ is \emph{asymmetric} in the sense that each object only admits the trivial automorphism. Hence, there are exactly $n!$ ways to label objects of size $n$ and it makes no difference whether we draw an object from $\scT_n^\omega$ proportional to its $\omega$-weight or any of the $n!$ labelled counterparts from the set of all labelled objects in $L(\scT)_n^\omega$ proportional to their $\omega$-weights. Further, the ordinary and exponential generating series coincide in this case. Keeping that comment in mind, the assumptions of Lemma \ref{lem:giant_component} are fulfilled. Denote by $\cT_n^{\min}$ and $\cT_n^{\max}$ ``the'' smallest and the largest tree corresponding to $\cT_n^*$. Whenever both trees have the same size, we pick an arbitrary order. Then, Lemma \ref{lem:giant_component} states that upon removing the largest component of $(\cT_n^{\min},\cT_n^{\max})$ (which is $\cT_n^{\max}$ per definition), the remainder converges in total variation to the random unlabelled planted tree $\cT$ given by
\[
	\Pr{\cT = T} = \frac{\omega(T) \rho_{\scT^{\omega}}^{\lvert T\rvert}}{\scT^\omega(\rho_{\scT^{\omega}})},
	\qquad T\in \scT^{\omega}.
\]
By Equation \eqref{eq:taq} we know that $\cT$ follows the distribution of a Galton-Watson tree with offspring distribution given in Equation \eqref{eq:offsp}. This fully proves Lemma \ref{le:step2}.

\subsection{Approximating unlabelled by labelled trees (Proof of Lemma \ref{le:step1})}
Let $\cL_n\sim L(\scU)^{\overline{\omega}}$ denote the random tree drawn proportional to its $\overline{\omega}$-weight from all \emph{labelled} unrooted plane trees of size $n$. The crucial step towards proving Lemma \ref{le:step1} is showing that, for large $n$, the unlabelled object $\cU_n$ can be approximated by drawing $\cL_n$ and dropping the labels afterwards. For any labelled tree $T$, let $\tilde{T}$ be its unlabelled counterpart.
\begin{lemma}
	\label{le:approx_unlab_lab_step1}
	There exist constants $C,c>0$ such that for all $n$
	\[
		d_{\mathrm{TV}}(\cU_n, \tilde{\cL}_n) 
		\le C \exp(-cn).
	\]
\end{lemma}
Any labelled unrooted plane tree of size $n$ has exactly $2(n-1)$ corners and by determining one corner as root and ordering the descendants in a canonical way we obtain a labelled planted tree. Further, as mentioned before, planted trees do only allow the trivial automorphism, by which we deduce that there are exactly $n!$ possibilities to obtain an unlabelled planted tree, as depicted in Figure \ref{fi:planted_plane}. All in all, this means that 
\[
	\tilde{\cL}_n
	\overset{(d)}{=} \tilde{\cT}_n^{*}
\]
and Lemma \ref{le:step1} follows readily. 
\begin{figure}
\centering
  \def\svgwidth{0.9\columnwidth}
	\resizebox{0.9\textwidth}{!}{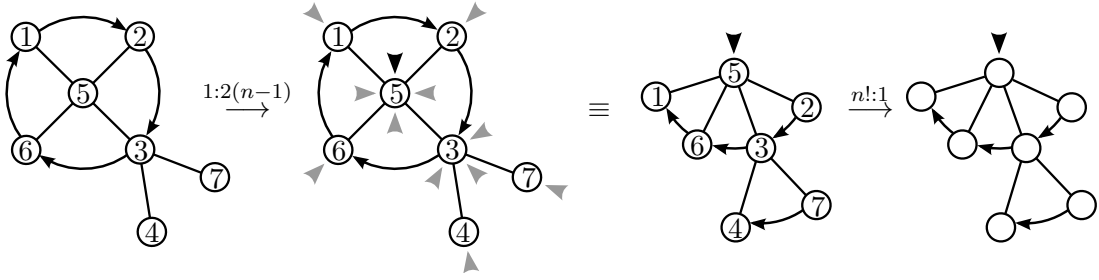}
		\caption{Given a labelled plane tree with $n$ vertices (first), there are $2(n-1)$ possibilities to appoint a corner as root (second). The position of the corner induces a natural ordering of the children of the root (third) and, as there exist no symmetries, we find $n!$ corresponding unlabelled trees (fourth). }
		\label{fi:planted_plane}
\end{figure}
\begin{proof}[Proof of Lemma \ref{le:approx_unlab_lab_step1}]

Let $\mathrm{Sym}(L(\scU)^{\overline{\omega}})$ be the set of all pairs $(L,\sigma)$, where $L$ is a labelled unrooted plane tree and $\sigma$ is an automorphism of $L$, that is $\sigma . L = L$. For any such object we define its $\overline{\omega}$-weight by
\[
	\overline{\omega}((L,\sigma)) 
	:= \overline{\omega}(L).
\]
Further, the (random) pair $(\cL_n,\sigma_n)$ is drawn proportional to its $\overline{\omega}$-weight from all symmetries of size $n$ in $\mathrm{Sym}(L(\scU)^{\overline{\omega}})$. For $(L,\sigma)\in \mathrm{Sym}(L(\scU)^{\overline{\omega}})$ we denote by $(L,\sigma)^{\sim}$ the tree which is obtained by removing the labelling of $L$ and discarding the automorphism. By the previous notation this means $(L,\sigma)^{\sim} := \tilde{L}$. It is a well known fact that every unlabelled structure in $\scU_n$ induces $n!$ symmetries in $\mathrm{Sym}(L(\scU)^{\overline{\omega}})$ in the sense that for any $U\in\scU_n$ there are $n!$ symmetries $(L,\sigma)$ such that $(L,\sigma)^{\sim} = U$, see Lemma \ref{le:unlabelled_n!_symmetries_and_ogs_in_cycle_index_sum}. This implies
\[
	\cU_n
	\overset{(d)}{=} (\cL_n,\sigma_n)^{\sim}.
\]
There is a canonical way of rooting any tree at its geometric center by pruning its leafs until there is only one edge or one node left. We call such trees \emph{edge-centered} and \emph{vertex-centered}, respectively. With this comment in mind, there appear three different scenarios for any $(L,\sigma)$ in $\mathrm{Sym}(L(\scU)^{\overline{\omega}})$:
\begin{enumerate}
\item $\sigma = \mathrm{id}$.
\item $L$ is vertex-centered and $\sigma \neq \mathrm{id}$.
\item $L$ is edge-centered and $\sigma \neq \mathrm{id}$.
\end{enumerate}

Symmetries being part of the first case can equally be seen as regularly labelled objects without an automorphism. \\
Whenever the tree of a symmetry $(L,\sigma)$ is vertex-centered and the automorphism is not trivial, there is a unique way of decomposing $L$, cf. Figure \ref{fi:vertex_cent}. Consider the cyclically ordered list of trees spreading from the center $v$ of $L$ and appoint each of the nodes connected to $v$ by an edge as root of the tree it is contained in. Let us call these planted trees $T_1,\dots,T_k$. As planted trees do only allow the trivial automorphism and $\sigma \neq \mathrm{id}$ each node in $T_1$ needs to be transported into some $T_\ell$ for $\ell\neq 1$ along $\sigma$. Hence, in order to maintain the linear ordering within $T_1$, \emph{all} of its vertices are send to $T_\ell$ and consequently $T_\ell$ is an identical copy of $T_1$. Similarly, all vertices of $T_\ell$ are mapped into either $T_1$ or some other identical tree. Continuing this procedure, we receive $d\ge 2$ identical trees, which interchange their entire set of vertices among each other under $\sigma$. By repeating the previous steps successively we may eventually partition $\{T_1,\dots,T_k\}$ into $j\in[k]$ subsets of sizes $d_1,\dots,d_j$, where each subset contains identical trees interchanging their vertex-sets under $\sigma$ and $d_1 + \cdots + d_j = k$. Recall that the roots of the $T_1,\dots,T_k$ are cyclically ordered around the center $v$, which immediately implies that $d_1 = \cdots = d_j$ and $j= k/d_1$ needs to be a divisor of $k$. Further, the only fixpoint in $\sigma$ sends the center $v$ to itself. \\
On the other hand, given the product of a single vertex $v$ and a cyclically ordered list of planted trees $T_1,\dots,T_k$ together with a non-$\mathrm{id}$ automorphism $\tilde{\sigma}$ (as before), we may reconstruct a vertex-centered tree by connecting $v$ to the roots of the planted trees. This is due to the fact that the longest path of two identical copies attached to $v$ by two edges is always odd. Then, by extending $\tilde{\sigma}$ with a fixpoint at $v$ we obtain a symmetry of the second case. \\
This decomposition guarantees that the set containing all symmetries of the second case, denoted by $\scR_v$, can be expressed as 
\[
	\scR_v 
	= \{(L,\sigma): L\in \scX \cdot ( \Cyc^{\kappa}\circ L(\scT)^{\omega} ) , ~\sigma.L = L,~ \sigma\neq \mathrm{id}\},
\]
where $\kappa(\tau)=\omega_{k-1}$ for $\tau\in\Cyc^{\kappa}_k$ and $\scX$ is the unique class consisting of a single vertex receiving weight 1 accounting for the center $v$.
As described before, we deduce that in this case any automorphism $\sigma$ contains only \emph{one} single fixpoint (sending the center $v$ to itself) as for otherwise \emph{all} cycles in $\sigma$ were fixpoints violating the assumption $\sigma \neq \mathrm{id}$. The cycle index sum $Z_{\Cyc^{\kappa}\circ L(\scT)^{\omega}}(0,x^2,x^3,\dots)$ counts symmetries of $\Cyc^{\kappa}\circ L(\scT)^{\omega}$ without any fixpoints and hence, according to \eqref{eq:gen_fct_product} and \eqref{eq:weighted1},
\[
	\cR_v(x)
	:= xZ_{\Cyc^{\kappa}}(Z_{L(\scT)^{\omega}}(0,x^2,x^3,\dots), Z_{L(\scT)^{\omega^2}}(x^2,x^4,x^6,\dots), \dots)
\]
counts exactly the symmetries in $\scR_v$. Per definition
\[
	Z_{L(\scT)^{\omega}}(0,s_2,s_3,\dots)
	= \sum_{k\ge 0} \frac{1}{k!} 
	\sum \omega(T) s_2^{\sigma_2}\cdots s_k^{\sigma_k},
\]
where the sum is conducted over all ${(T,\sigma)\in \mathrm{Sym}(L(\scT)^{\omega})}$ such that $\sigma$ has no fixpoint. On the other hand, labelled planted trees do only allow the trivial automorphism and thus $Z_{L(\scT)^{\omega}}(0,x^2,x^3,\dots)=0$. This shows
\[
	x^{-1}\cR_v(x)
	= Z_{\Cyc^{\kappa}}(0,\scT^{\omega^2}(x^2),\scT^{\omega^3}(x^3),\dots)
	= \sum_{d\ge 2} \frac{\varphi(d)}{d} \sum_{j\ge 1} \frac{\omega_{jd-1}}{j} 
	\left( \scT^{\omega^d}( x^d) \right)^j.
\]
Note that we used $\widetilde{L(\scT)^{\omega^i}}(x) = \scT^{\omega^i}(x)$ in the latter identity.

\begin{figure}[t]
	\centering
	\centering
	\includegraphics[width=0.7\textwidth]{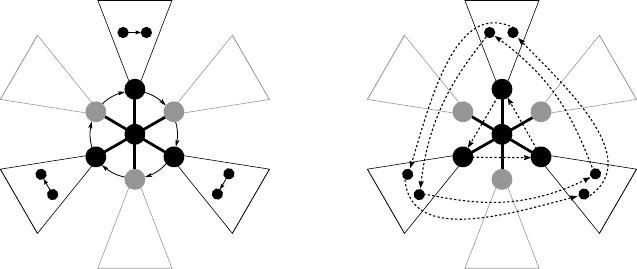}
	\caption{On the left: a vertex-centered plane tree, where solid arrows indicate the ordering of vertices within the tree; on the right: dashed arrows indicate a non-$\mathrm{id}$ permutation which preserves the ordering and sends all vertices of each tree to their respective clones}
	\label{fi:vertex_cent}
\end{figure} 

For the third case consider $(L,\sigma)$ such that $L$ is edge-centered and $\sigma\neq id$. By similar reasoning as for the vertex-centered case, $L$ consists of two copies of the same planted tree (having half the size of $L$) connected by an edge and the automorphism $\sigma$ sends each vertex in one tree to its clone in the other tree, see Figure \ref{fi:edge_cent}. Clearly, the respective generating series is
\[
	\cR_e(x)
	:= \scT^{\omega^2}(x^2).
\]
\begin{figure}[t]
	\centering
		\centering
		\includegraphics[width=0.7\textwidth]{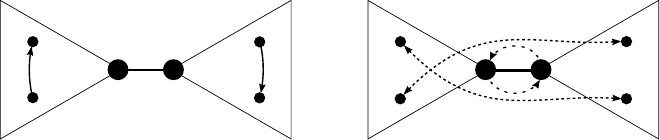}
		\caption{On the left: an edge-centered plane tree, where solid arrows indicate the ordering of vertices within the tree; on the right: dashed arrows indicate a non-$\mathrm{id}$ permutation which preserves the ordering and sends all vertices of each tree to their respective clones}
		\label{fi:edge_cent}
\end{figure}
Concluding we have that
\[
	Z_{\scU^{\overline{\omega}}}(x,x^2,x^3,\dots)
	= L(\scU)^{\overline{\omega}}(x) + \cR_v(x) + \cR_e(x),
\]
where $L(\scU)^{\overline{\omega}}(x)$ is the exponential generating series of labelled plane trees.  \\
Define $\cR(x) = \cR_v(x) + \cR_e(x)$ and assume $n \equiv 2 \mod \spa(\mathbf{w})$. Given the random structure $(L_n,\sigma_n)$ drawn proportional to its $\overline{\omega}$-weight from $\mathrm{Sym}(L(\scU)^{\overline{\omega}})$ we first observe
\[
	\Pr{ \sigma_n \neq \mathrm{id}_n} 
	=  \frac{[x^n] \cR(x)}{[x^n](L(\scU)^{\overline{\omega}}(x)+\cR(x))}
	\le \frac{[x^n] \cR(x)}{[x^n]L(\scU)^{\overline{\omega}}(x)}.
\]
Again we make use of the fact that there is an $n!:2(n-1)$-relation between labelled plane trees and unlabelled planted trees to obtain
\[
	[x^n]L(\scU)^{\overline{\omega}}(x)
	= \frac{1}{2(n-1)} [x^n] \scT^{\overline{\omega}}(x)
	= \frac{1}{2(n-1)} [x^n] \scT^\omega(x)^2,
\]
where the latter identity is due to Equation \eqref{eq:relation_weighting_gen_series_planted}. By Lemma \ref{le:partf} the shifted generating series $x^{-1}\scT^\omega(x)$ is subexponential for $d = \spa(\mathbf{w})$. Hence, Lemma \ref{lem:subexp_composition} gives us the existence of a constant $A>0$ such that
\[
	[x^n] x^{-2} \scT^\omega(x)^2
	\ge A \cdot [x^n]x^{-1}\scT^\omega(x)
\]
for all $n\equiv 0 \mod d$. We conclude
\[
	\frac{[x^n] \cR(x)}{[x^n]L(\scU)^{\overline{\omega}}(x)}
	= 2(n-1) \frac{[x^{n-2}] x^{-2} \cR(x)}{[x^{n-2}] x^{-2}\scT^{\omega}(x)}
	\le A \cdot 2(n-1) \frac{[x^{n-2}] x^{-1}\cR(x)}{[x^{n-2}]x^{-1}\scT^{\omega}(x)}
\]
for all $n\equiv 2 \mod d$. If we assume that the radius of convergence of $x^{-2}\cR(x)$ is strictly greater than the one of $x^{-2}\scT^{\omega}(x)$, Lemma \ref{le:coeff_comparison_subexp} gives us the existence of $C,c>0$ such that, for $n\equiv 2\mod d$,
\[
	\frac{[x^{n-2}] x^{-2}\cR(x)}{[x^{n-2}]x^{-2}\scT^{\omega}(x)}
	\le C \exp(-cn).
\]
Consequently, for some $C',c'>0$,
\[
	\Pr{ \sigma_n \neq \mathrm{id}_n} 
	\le A \cdot 2(n-1)\frac{[x^{n-2}] x^{-2}\cR(x)}{[x^{n-2}]x^{-2}\scT^{\omega}(x)}
	\le C' \exp(-c'n)
\]
and Lemma \ref{le:approx_unlab_lab_step1} is proven. \\
Hence, to finish the proof, it suffices to show that
\begin{align}
	\label{eq:rho_S<max}
	\rho_{\scT^{\omega}} < \rho_{\cR_e} < \rho_{\cR_v}.
\end{align}
Therefore, we follow closely the presentation in \cite[Sec. 7]{MR2908619} and apply the related results given there. Define
\[
	\Phi_d(x)
	:= \sum_{k\ge 0} \omega_k^d x^k
\]
together with its radius of convergence $\rho_d$ for $d\ge 2$. Then the recursive structure of planted trees guarantees $\scT^{\omega^d}(x) = x \Phi_d( \scT^{\omega^d}(x))$. 
Next, let $\Psi$ and $\nu$ be defined as in \eqref{eq:psi_nu}. We recall that there appear two cases: if $\nu \ge 1$, let $\tau$ be the unique number in $[0,\rho]$ such that $\Psi(\tau)=1$, and set $\tau:=\rho$ otherwise. Likewise, we define $\nu_d$ and $\tau_d$ with respect to $\Phi_d$. Alternatively, in \emph{both} cases $\tau$ can be characterized as the unique maximum point of $\Phi(t)$ on the interval $[0,\rho]$ and an equivalent characterization holds for $\tau_d$ as well, cf. \cite[Remark 7.4]{MR2908619}. The radii of convergence of $\scT^{\omega}(x)$ and $\scT^{\omega^d}(x)$ fulfil
\[
	\rho_{\scT^{\omega}} = \frac{\tau}{\Phi(\tau)}
	\quad
	\text{and}
	\quad
	\rho_{\scT^{\omega^d}} = \frac{\tau_d}{\Phi_d(\tau_d)}
\]
as well as
\begin{align}
	\label{eq:tau=T(rho)}
	\tau
	= \scT^{\omega} (\rho_{\scT^{\omega}})
	\quad
	\text{and}
	\quad
	\tau_d
	= \scT^{\omega^d} (\rho_{\scT^{\omega^d}}),
	\quad
	d\ge 2,
\end{align}
as stated in \eqref{eq:taq} and \cite[Remark 7.5]{MR2908619}. By standard tilting arguments, we might assume that $(\omega_k)_{k\ge 0}$ is a probability weight sequence (which implies that $\rho \ge 1$) and $\rho_{\scT^\omega}=1$, see Remark \ref{rem:tilting}. We know that $\rho$ can be computed by
\[
	\rho 
	= 1/ \limsup_{k\to\infty} \omega_k^{1/k}.
\]
Hence,
\[
	\rho_d 
	= 1/ \limsup_{k\to\infty} \omega_k^{d/k}
	= \rho^d
\] 
and $\rho_d \ge 1$ for all $d\ge 2$. Consequently, the choice of $\tau_d$ yields
\[
	\rho_{\scT^{\omega^d}}
	\ge \Phi_d(1)^{-1},
\]
giving us that the radius of convergence of $\scT^{\omega^d}(x^d)$ satisfies
\[
	\rho_{\scT^{\omega^d}}^{1/d}
	\ge \Phi_d(1)^{-1/d}
	= \lVert \omega \rVert_d^{-1}
	\ge \lVert \omega \rVert_2^{-1}
	> 1
	= \rho_{\scT^{\omega}}.
\]
With latter inequality at hand, there exist $\varepsilon>0$ and $0<\lambda<1$ such that
\[
	\rho_{\scT^{\omega}} + \varepsilon
	\le \lambda \rho_{\scT^{\omega^d}}^{1/d}
\]
and consequently with \eqref{eq:tau=T(rho)}
\begin{align*}
	\scT^{\omega^d}((\rho_{\scT^{\omega}} + \varepsilon)^d)
	\le \scT^{\omega^d}( \lambda^d \rho_{\scT^{\omega^d}}) 
	\le \lambda^d \scT^{\omega^d}( \rho_{\scT^{\omega^d}}) 
	= \lambda^d \tau_d
	\le \lambda^d \rho_d
	= \lambda^d \rho^d
\end{align*}
uniformly in $d\ge 2$. Thus,
\begin{align}
	\begin{split}
		\label{eq:rv_d>rho}
		(\rho_{\scT^{\omega}} + \varepsilon)^{-1}\cR_v(\rho_{\scT^{\omega}} + \varepsilon)
		&\le \sum_{d\ge 2} \frac{\varphi(d)}{d} \sum_{j\ge 1} \frac{\omega_{jd-1}}{j} 
		\lambda^{jd}\rho^{jd} \\
		&\le \sum_{m \ge 2} \omega_{m-1} \rho^m \frac{\lambda^m}{m} \sum_{jd=m} \varphi(d) \\
		&\le \sum_{m \ge 2} \omega_{m-1} \rho^m m\lambda^m
		< \infty,
	\end{split}
\end{align}
by which it is proven that $\rho_{\scT^{\omega}} < \rho_{\cR_v}$. Finally, it is straight-forward to understand that $\rho_{\cR_v}<\rho_{\cR_e}$ and the claim \eqref{eq:rho_S<max} follows.
\end{proof}


\section{Proofs of the applications}
\label{sec:profapp}

\subsection{Scaling limits and tail bounds for the diameter}

We will use the following easy observation, which is inspired by a result for the Gromov--Hausdorff metric \cite[Thm. 7.3.25]{MR1835418}.
\begin{lemma}
	\label{le:ghp}
	Let $(X, d_X, \mu_X)$ be a measured compact metric space. Then for all $\beta>0$ it holds that
	\[
	d_{\textsc{GHP}}((X, \mu_X), (\beta X, \mu_X)) \le |1 - \beta| \Di(X)/2.
	\]
\end{lemma}
\begin{proof}
	In order to avoid confusion, set $(Y, d_Y) = (X, \beta d_X)$. For each $x \in X$ let $x'$ denote the corresponding point in $Y$.
	Consider the metric on the disjoint union $X \sqcup Y$ that extends the metrics on $X$ and $Y$, and satisfies for all $x_1 \in X$ and $x_2' \in Y$
	\[
	d(x_1,x_2') = |1- \beta| \Di(X)/2 + \min( d_X(x_1, x_2), d_Y(x_1', x_2')).
	\]
	It is elementary to verify that $d$ satisfies the axioms of a metric, and that
	\[
	d_{\textsc{P}}(\mu_X,\mu_Y) = d_{\textsc{H}}(X,Y) = |1- \beta| \Di(X)/2.
	\]
\end{proof}

We are now ready to present the proof for the scaling limit and diameter tail-bounds for the random tree $\cU_n$.

\begin{proof}[Proof of Theorem~\ref{te:scaling}]
	In the finite variance case, the simply generated planted plane tree $\cT_m$ satisfies a scaling limit as in \eqref{eq:ap1} by results of Aldous~\cite[Thm. 23]{MR1207226}. (The case with periodic offspring distributions was  stated in Le Gall {\cite[Thm 6.1]{MR2603061}}.) It satisfies a tail bound as in \eqref{eq:ap2} by results of Addario-Berry, Devroye and Janson~\cite[Thm. 1.2]{MR3077536}.

	To verify the tail bound~\eqref{eq:ap2}, let $C_1,c_1>0$ be constants such that for all $m$ and $x \ge 0$
	\begin{align}
	\label{eq:tmp1}
	\Pr{ \cT_m \ge x} \le C_1 \exp(-c_1 x^2/m).
	\end{align}
	Lemma~\ref{le:step1}  yields that that there are constants $C_2, c_2>0$ such that
	\[
	\Pr{\Di(\cU_n) \ge x} \le C_2 \exp(- c_2 n) + \Pr{\Di(\cT_n^*) \ge x}. 
	\]
	If $x>n$, then the left-hand side of this inequality equals zero. If $x\le n$, then it clearly holds that $\exp(- c_2 n) \le \exp(-c_2 x^2/n)$. Hence we may write
	\begin{align}
	\label{eq:tmp2}
	\Pr{\Di(\cU_n) \ge x} \le C_2 \exp(- c_2 x^2/n) + \Pr{\Di(\cT_n^*) \ge x}. 
	\end{align}
	The diameter of the tree $\cT^*$ may be bounded by twice its height. If $\cT^*$ has large height $\He(\cT^*) \ge x/2$, then it holds that $\He(\cT^{(1)}) \ge x/2 -1$ or  $\He(\cT^{(2)}) \ge x/2 -1$. Both events have the same probability. Setting $\bar{x}=x/2-1$, it follows from Lemma~\ref{le:step2} and Inequality~\eqref{eq:tmp1} that
	\begin{align*}
	\Pr{\Di(\cT_n^*) \ge x} &\le 2 \Pr{\He(\cT_n^{(1)}) \ge \bar{x}} \\
	&= 2 \sum_{0 \le k \le n} \Pr{|\cT_n^{(1)}| = k} \Pr{\He(\cT_k) \ge \bar{x}} \\
	&\le 2 C_1 \sum_{0 \le k \le n} \Pr{|\cT_n^{(1)}| = k} \exp(-c_1 \bar{x}^2/k) \\
	&\le 2 C_1 \exp(-c_1 \bar{x}^2/n).
	\end{align*}
	Hence the bound~\eqref{eq:ap2} follows from Inequality~\eqref{eq:tmp2}. 
	
	We proceed with establishing the scaling limit for the random tree $\cU_n$. For each $m$, let $\mu_m$ denote the uniform measure on the leaves of $\cT_m$. By Theorem~\ref{te:main} it suffices to show convergence for the random tree $\mS_n$ with the uniform measure $\eta_n$ on its leaves. Clearly $K_n = n + O_p(1)$ implies that
	\begin{align}
	\label{eq:yep}
	\left( \frac{\sigma \cT_{K_n}}{2 \sqrt{K_n}}, \mu_{K_n} \right)  \convdis (\cT_{\mathrm{Br}}, \mu_{\mathrm{Br}} )
	\end{align}
	in the Gromov--Hausdorff--Prokhorov sense as $n$ becomes large. Hence it suffices to show that
	\begin{align}
	\label{eq:s1}
	d_{\textsc{GHP}}\left(  \left (\frac{\sigma \cS_n}{2 \sqrt{n}}, \eta_n \right ), \left (\frac{\sigma \cT_{K_n}}{2 \sqrt{n}}, \mu_{K_n} \right ) \right) \convp 0
	\end{align}
	and
	\begin{align}
	\label{eq:s2}
	d_{\textsc{GHP}}\left(  \left (\frac{\sigma \cT_{K_n}}{2 \sqrt{n}}, \mu_{K_n} \right ), \left (\frac{\sigma \cT_{K_n}}{2 \sqrt{K_n}}, \mu_{K_n} \right ) \right) \convp 0.
	\end{align}
	
	We start with \eqref{eq:s1}.
	Note that the Hausdorff distance $d_{\textsc{GH}}(\cS_n, \cT_{K_n})$ between $\cS_n$ and $\cT_{K_n}$ is bounded by $n - K_n = O_p(1)$. Thus,
	\begin{align}
	\label{eq:nub}
	d_{\textsc{H}} \left (\frac{\sigma \cS_n}{2 \sqrt{n}}, \frac{\sigma \cT_{K_n}}{2 \sqrt{n}} \right ) = \frac{\sigma}{2 \sqrt{n}} d_{\textsc{H}}(\cS_n, \cT_{K_n}) =\frac{\sigma}{2 \sqrt{n}} O_p(1) \convdis 0
	\end{align}
	as $n$ becomes large. The Prokhorov distance $d_{\textsc{P}}(\cdot, \cdot)$ is not homogeneous, so we have to argue differently. For each tree $T$ let $L(T)$ denote its set of leaves. If the tree is rooted, then it is custom to never count the root as a leaf, unless its the only vertex of the tree.  If we regard $\eta_n$ and $\mu_{K_n}$ as measures on the set of vertices of the tree $\cS_n$, then for each subset $A \subset \cS_n$
	\[
	\eta_n(A) = |L(\cS_n) \cap A| / |L(\cS_n)| \qquad \text{and} \qquad \mu_{K_n}(A) =  |L(\cT_{K_n}) \cap A| / |L(\cT_{\mK_n})|.
	\]
	The nominator of the first quotient differs from the nominator of the second by at most $n - K_n$, and the same clearly holds for the denominator.
	It is elementary, that consequently the total variational distance of the random measures $\eta_n$ and $\mu_{K_n}$ may be bounded by
	\[
	d_{\textsc{TV}}(\eta_n, \mu_{K_n}) \le 2 (n - K_n) / |L(\cT_{K_n})|.
	\]
	Janson~\cite[Thm. 7.11]{MR2908619} showed that
	\[
	L(\cT_m) / m \convp \Pr{\xi =0} > 0
	\]
	as $m$ becomes large. As $n - K_n = O_p(1)$, it follows that
	\[
	d_{\textsc{P}}\left(  \left (\frac{\sigma \cS_n}{2 \sqrt{n}}, \eta_n \right ), \left(\frac{\sigma \cT_{K_n}}{2 \sqrt{n}}, \mu_{K_n} \right) \right) \le d_{\textsc{TV}}(\eta_n, \mu_{K_n}) \le \frac{2(n - K_n)}{K_n( \Pr{\xi=0} + o_p(1))}.
	\]
	This bound clearly converges in probability to $0$ as $n$ becomes large. Together with \eqref{eq:nub}, the limit \eqref{eq:s1} readily follows. 
	
	In order to verify \eqref{eq:s2}, note that by the limit \eqref{eq:yep} it holds that \[\Di(\cT_{K_n}) / \sqrt{K_n} = O_p(1).\] It follows by Lemma~\ref{le:ghp} that
	\begin{align*}
	d_{\text{GHP}}\left( \left( \frac{\sigma \cT_{K_n}}{2 \sqrt{n}} \mu_{K_n} \right), \left( \frac{\sigma \cT_{K_n}}{2 \sqrt{K_n}}, \mu_{K_n} \right) \right) &\le \frac{\sigma \Di(\cT_{K_n})}{2} \left( \frac{1}{\sqrt{n}} - \frac{1}{\sqrt{K_n}} \right) \\
	&\le O_p(1)\left( \sqrt{\frac{K_n}{n}} -1 \right).
	\end{align*}
	Since $K_n = n +O_p(1)$, this bound clearly converges in probability to $0$ as $n$ becomes large. Thus \eqref{eq:s2} holds.
	
	
	In the infinite variance setting of claim 2., the scaling limit \eqref{eq:ap3} holds for the simply generated planted plane tree $\cT_m$ by results of Duquesne~\cite{MR1964956} (see also Kortchemski~\cite{MR3185928} and Miermont and Haas~\cite{MR3050512}), and the tail bound~\eqref{eq:ap4} by Kortchemski~\cite{2015arXiv150404358K}. This allows us to deduce the tail bound and convergence for $\cU_n$ in an identical way as for the finite variance case. In order to avoid redundancy, we will not make this explicit.
\end{proof}

\subsection{Benjamini-Schramm limits and degree distribution}

\begin{proof}[Proof of Theorem~\ref{te:bslimit}]
	We start with the graph limit. By Theorem~\ref{te:main}, it suffices to show convergence for the random tree $\mS_n$. Let $v_n$ be a uniformly at random drawn vertex of the tree $\mS_n$, and likewise $u_{K_n}$ a uniformly sampled node of $\cT_{K_n}$. Aldous~\cite{MR1102319} showed that
	\[
	(\cT_m, u_m) \convdis \cT^\bullet
	\]
	with respect to the local metric $d_{\textsc{L}}$ as $m$ becomes large. Since $K_n = n + O_p(1)$, it follows that also
	\[
	(\cT_{K_n}, u_{\cK_n}) \convdis \cT^\bullet.
	\]
	Thus it suffices to verify that
	\begin{align}
	\label{eq:ltoshow}
	d_{\textsc{L}}((\cS_n, v_n), (\cT_{K_n}, u_{\cK_n})) \convp 0
	\end{align}
	as $n$ tends to infinity. Since the compliment of the tree $\cT_{K_n}$ in $\cS_n$ has stochastically bounded size, it follows that the random vertex $v_n$ lies in the tree $\cT_{K_n}$ with probability tending to $1$ as $n$ becomes large. Conditioned on belonging to $\cT_{K_n}$, the vertex $v_n$ is distributed like the uniform node $u_{K_n}$. For all $\ell \ge 0$ let $Z_\ell(\cT_m)$ denote the number of vertices with distance $\ell$ from the root in $\cT_m$. Janson~\cite[Thm. 1.13]{MR2245498} showed that for each $r \ge 1$ there is a constant $C>0$ with
	\[
	\Ex{Z_\ell(\cT_m)^r} \le C \ell^r
	\]
	for all $m$ and $\ell$. In particular for each fixed $\ell$ it holds that $U_\ell(\cT_m)$ is stochastically bounded as $m$ becomes large. Consequently, the same holds for $U_\ell(\cT_{K_n})$. Hence the random vertex $u_{K_n}$ lies outside of $U_\ell(\cT_{K_n})$ with probability tending to $1$ as $n$ becomes large. Whenever this is the case, it follows that
	\[
	U_\ell(\cS_n, u_{K_n}) = U_\ell(\cT_{K_n}, u_{K_n}).
	\]
	Consequently, 
	\[
	d_{\textsc{TV}}( U_\ell(\cS_n, v_n), U_\ell(\cT_{K_n}, u_{K_n})) \to 0
	\]
	as $n$ becomes large. As this holds for arbitrarily large fixed $\ell$, the limit~\eqref{eq:ltoshow} follows.
	
	It remains to prove the central limit theorem for the number $N_d$ of nodes with degree $d$ in $\cU_n$. For each $m$, let $\bar{N}_{d}(\cT_m)$ denote the number of nodes with out-degree $d$ in the tree $\cT_m$. Kolchin~\cite[Thm. 2.3.1]{MR865130} showed that 
	\[
	\frac{\bar{N}_d(\cT_m) - \Pr{\xi=d}m}{\sqrt{m}} \convdis \cN(0, \sigma_d^2)
	\]
	as $m$ becomes large for some $\sigma_d >0$. Since $K_n = n + O_p(1)$, it follows that
	\[
	\frac{\bar{N}_d(\cT_{K_n}) - \Pr{\xi=d}n}{\sqrt{n}} = \frac{\bar{N}_d(\cT_{K_n}) - \Pr{\xi=d}K_n}{\sqrt{K_n}} + o_p(1)\convdis \cN(0, \sigma_d^2).
	\]
	As the size of the small tree attached to the root of $\cT_{K_n}$ in $\cS_n$ is bounded, it follows that
	\[
	|N_d(\cU_n) - \bar{N}_{d-1}(\cT_{K_n})| = O_p(1).
	\]
	Consequently,
	\[
	\frac{N_d(\cU_n) - \Pr{\xi=d-1}n}{\sqrt{n}} \convdis \cN(0, \sigma_{d-1}^2).
	\]
\end{proof}

\bibliographystyle{siam}
\bibliography{uplane}

\end{document}